\documentclass[english,12pt]{smfart}
\usepackage{amssymb}
\usepackage{amsfonts}
\usepackage{amscd}
\usepackage[all]{xypic}

\CompileMatrices

\swapnumbers

\def\ord{{\rm ord}}

\def\11{{\mathbb 1}}
\def\AA{{\mathbb A}}

\def\CC{{\mathbb C}}

\def\FF{{\mathbb F}}
\def\GG{{\mathbb G}}

\def\NN{{\mathbb N}}

\def\QQ{{\mathbb Q}}
\def\RR{{\mathbb R}}

\def\ZZ{{\mathbb Z}}

\def\cI{{\mathcal I}}

\def\cM{{\mathcal M}}

\def\cO{{\mathcal O}}

\mathchardef\alphag="7C0B \mathchardef\betag="7C0C
\mathchardef\gammag="7C0D \mathchardef\deltag="7C0E
\mathchardef\varepsilong="7C22 \mathchardef\varphig="7C27
\mathchardef\psig="7C20 \mathchardef\zetag="7C10
\mathchardef\epsilong="7C0F \mathchardef\rhog="7C1A
\mathchardef\taug="7C1C \mathchardef\upsilong="7C1D
\mathchardef\iotag="7C13 \mathchardef\thetag="7C12
\mathchardef\pig="7C19 \mathchardef\sigmag="7C1B
\mathchardef\etag="7C11 \mathchardef\omegag="7C21
\mathchardef\kappag="7C14 \mathchardef\lambdag="7C15
\mathchardef\mug="7C16 \mathchardef\xig="7C18
\mathchardef\chig="7C1F \mathchardef\nug="7C17
\mathchardef\varthetag="7C23 \mathchardef\varpig="7C24
\mathchardef\varrhog="7C25 \mathchardef\varsigmag="7C26
\mathchardef\Omegag="7C0A \mathchardef\Thetag="7C02
\mathchardef\Sigmag="7C06 \mathchardef\Deltag="7C01
\mathchardef\Phig="7C08 \mathchardef\Gammag="7C00
\mathchardef\Psig="7C09 \mathchardef\Lambdag="7C03
\mathchardef\Xig="7C04 \mathchardef\Pig="7C05
\mathchardef\Upsilong="7C07

\newtheorem{theorem}[subsection]{Theorem}
\newtheorem{theorems}[subsubsection]{Theorem}
\newtheorem{lem}[subsection]{Lemma}
\newtheorem{cor}[subsection]{Corollary}
\newtheorem{prop}[subsection]{Proposition}

\theoremstyle{definition}

\newtheorem{def-prop}[subsection]{Proposition-Definition}
\newtheorem{def-theorem}[subsection]{Theorem-Definition}
\newtheorem{def-lem}[subsection]{Lemma-Definition}

\theoremstyle{remark}

\theoremstyle{plain}

\numberwithin{equation}{subsection}

\def\boxit#1#2{\setbox1=\hbox{\kern#1{#2}\kern#1}%
\dimen1=\ht1 \advance\dimen1 by #1 \dimen2=\dp1 \advance\dimen2 by
#1
\setbox1=\hbox{\vrule height\dimen1 depth\dimen2\box1\vrule}%
\setbox1=\vbox{\hrule\box1\hrule}%
\advance\dimen1 by .4pt \ht1=\dimen1 \advance\dimen2 by .4pt
\dp1=\dimen2 \box1\relax}

\newcommand{\sur}[2]{\genfrac{}{}{0pt}{}{#1}{#2}}
\renewcommand{\theequation}{\thesubsection.\arabic{equation}}

\DeclareMathOperator*{\grad}{grad}

\def\ord{{\rm ord}}

\def\llp{\mathopen{(\!(}}

\def\rrp{\mathopen{)\!)}}

\begin{document}

\author[Raf Cluckers]{Raf Cluckers}
\address{Katholieke Universiteit Leuven,
Departement wiskunde, Celestijnenlaan 200B, B-3001 Leu\-ven,
Bel\-gium. Current address: \'Ecole Normale Sup\'erieure,
D\'epartement de ma\-th\'e\-ma\-ti\-ques et applications, 45 rue
d'Ulm, 75230 Paris Cedex 05, France}
\email{raf.cluckers@wis.kuleuven.be}
\urladdr{www.dma.ens.fr/$\sim$cluckers/}

\thanks{During the realization of this project, the author was
a postdoctoral fellow of the Fund for Scientific Research - Flanders
(Belgium) (F.W.O.)}

\begin{abstract}
We prove the remaining part of the conjecture by  Denef and Sperber
[Denef, J. and Sperber, S., \textit{Exponential sums mod $p^n$ and
{N}ewton polyhedra}, Bull.~Belg.~Math.~Soc., {\bf{suppl.}} (2001)
55-63] on nondegenerate local exponential sums modulo $p^m$. We
generalize Igusa's conjecture of the introduction of [Igusa, J.,
\textit{Lectures on forms of higher degree}, Lect.~math.~phys.,
Springer-Verlag, {\bf{59}} (1978)] from the homogeneous to the
quasi-homogeneous case and prove the nondegenerate case as well as
the modulo $p$ case. We generalize some results by Katz of [Katz, N.
M., \textit{Estimates for "singular" exponential sums},  Internat.
Math. Res. Notices (1999)  no. 16, 875-899] on finite field
exponential sums to the quasi-homogeneous case.
\end{abstract}



\title
{Exponential sums: questions by Denef, Sperber, and Igusa}

\maketitle

\section{Introduction}

\subsection{Global sums: from homogeneous to quasi-homogeneous}
Let $f$ be a polynomial over $\ZZ$ in $n$ variables. As in
\cite{CDenSper}, \cite{Cigumodp} we look at the ``global''
exponential sums
$$
S_f(\frac{1}{N}):= \frac{1}{N^n}\sum_{x\in
\{0,\ldots,N-1\}^n}\exp(2\pi i\frac{ f(x)}{N}),
$$
where $N$ varies over the positive integers. In order to bound
$|S_f(1/N)|$, it suffices to find bounds when $N=p^m$ for integers
$m>0$ and prime numbers $p$.

When $f$ is homogeneous and nondegenerate w.r.t. its Newton
polyhedron at zero, Igusa's conjecture for the toric resolution of
$f$ conjectures that there exists $c>0$ such that
$$
|S_f(\frac{1}{p^m})|<cp^{-\sigma m}m^{n-1},
$$
for all primes $p$ and integers $m>0$, where $\sigma$ is the biggest
rational number such that $(1/\sigma,\ldots,1/\sigma)$ lies on the
Newton polyhedron at zero $\Delta_0(f)$ of $f$. In a later paper,
Denef and Sperber conjectured for the same $f$ that the stronger
bound
$$
|S_f(\frac{1}{p^m})|<cp^{-\sigma m}m^{\kappa-1},
$$
holds for some $c$, uniformly in big enough primes $p$ and all
integers $m>0$, with $\kappa$ the codimension in $\RR^n$ of the
smallest face of $\Delta_0(f)$ which contains
$(1/\sigma,\ldots,1/\sigma)$.
 Both these conjectures are proved, by Denef and Sperber
\cite{DenSper} under the extra condition that no vertex of
$\Delta_0(f)$ belongs to $\{0,1\}^n$ and by the author
\cite{CDenSper} in general.

\par
In this paper, we show that both these conjectures also hold when
$f$ is quasi-homogeneous and nondegenerate w.r.t.~$\Delta_0(f)$.
Quasi-homogeneous means that there exist integers $a_i>0$ such that
$f(x_1^{a_1},\ldots,x_n^{a_n})$ is homogeneous. This gives evidence
to our conjecture of \cite{Cigumodp} that the analogon of Igusa's
conjecture for exponential sums of \cite{Igusa3} holds for all
quasi-homogeneous polynomials, and not only for homogeneous ones.

We give some more evidence for this conjecture by proving the
analogue of Igusa's conjecture (with the motivic oscillation index
instead of the above $\sigma$, as in \cite{Cigumodp}), for all
quasi-homogeneous polynomials (also degenerated ones) for $m=1$,
thereby generalizing \cite{Cigumodp} and some results by Katz of
\cite{Katz} to the quasi-homogeneous case.

This work generalizes most of the known evidence for Igusa's
conjecture to evidence for its generalization to quasi-homogeneous
polynomials. (To our knowledge, only the case of isolated
singularities of \cite{Igusa3} is only done for homogeneous
polynomials, and remains open for quasi-homogeneous ones.)

\subsection{Local sums}
Since for more general $f$ than quasihomogeneous $f$, Igusa's
conjecture is not even conjecturally understood, Denef and Sperber
designed a local variant of Igusa's conjecture which they conjecture
to hold for all polynomials, see  \cite{DenSper}. We treat only the
nondegenerate case. Let $g$ be a polynomial in $n$ variables which
is nondegenerate w.r.t.~its Newton polyhedron at zero $\Delta_0(g)$.
Denef and Sperber studied the ``local'' exponential sum
$$
T_g(\frac{1}{p^m}):=\frac{1}{p^{mn}}\sum_{x\in \{jp\mid
j=1,\ldots,p^{m-1}\}}\exp(2\pi i\frac{ g(x)}{p^m}),
$$
for $p$ a prime and $m>0$ an integer, and conjectured that there
exists $c$ such that
$$
|T_g( \frac{1}{p^m})|<cp^{-\sigma(g) m}m^{\kappa(g)-1}
$$
for big enough prime numbers $p$ and all integers $m>0$, with
$\sigma(g)$ and $\kappa(g)$ as $\sigma$ and $\kappa$ but for $g$
instead of for $f$. Denef and Sperber \cite{DenSper} proved this
under the condition that no vertex of $\Delta_0(g)$ belongs to
$\{0,1\}^n$. In this paper we prove this conjecture of
\cite{DenSper} for all polynomials $g$ which are nondegenerate
w.r.t.~$\Delta_0(g)$ (thereby removing the condition of
\cite{DenSper} that no vertex of $\Delta_0(g)$ belongs to
$\{0,1\}^n$).

\subsection{Global sums: Igusa's question mod $p$}
Theorem \ref{mt2} below is the modulo $p$ case for quasi-homogeneous
polynomials of Igusa's conjecture for exponential sums of
\cite{Igusa3}. In this theorem, the polynomial $h$ is not
necessarily nondegenerate. It generalizes the main result of
\cite{Cigumodp} from the homogeneous to the quasi-homogeneous case.
The modulo $p^2$ case is already proven in \cite{Cigumodp} for all
polynomials. In order to prove Theorem \ref{mt2} we generalize some
results by Katz \cite{Katz}, see Theorem \ref{Katzgen} below.

\section{A dictionary}

\subsection{From finite sums to $p$-adic integrals}

For $\QQ_p$ the field of $p$-adic numbers, $x=(x_1,\ldots,x_n)$
variables running over $\QQ_p^n$, let $|dx|$ be the unique
(real-valued) Haar measure on $\QQ_p^n$ so normalized that $\ZZ_p^n$
has measure one. For $a\in\QQ_p$, the complex number
$$
\exp(2\pi i a\bmod \ZZ_p):=\exp(2\pi i a')
$$
does not depend on the choice of representant $a'$ in $\QQ$ of
$a\bmod \ZZ_p$, and will be denoted by $\exp(2\pi i a)$.

Then, for $f$ a polynomial in $n$ variables over $\QQ$, one has the
equality
$$
S_f(\frac{1}{p^m})=\int_{\ZZ_p^n}\exp(2\pi i \frac{f(x)}{p^m})|dx|,
$$
and we will more generally consider for all $y\in \QQ_p$ the
integral
$$
S_{f,\QQ_p}(y):= \int_{\ZZ_p^n}\exp(2\pi i yf(x))|dx|.
$$

\subsection{From $\QQ_p$ to finite field extensions of $\QQ_p$ and to
$\FF_q\llp t\rrp $}\label{ST}

If $f$ is a polynomial over $\cO[1/N]$ for some ring of integers
$\cO$ and $N>0$ an integer, when $K$ is a nonarchimedean local field
which is an algebra over $\cO[1/N]$ (thus $p$-adic or of the form
$\FF_q\llp t\rrp $), and $\psi_K:K\to\CC^\times$ is a nontrivial
additive character which is $1$ on $\cO_K$ and different from $1$ on
some element of order $-1$, then we write for $y\in K$
$$
S_{f,K,\psi_K}(y):=\int_{\cO_K^n}\psi_K( y f(x))|dx|,
$$
with $\cO_K$ the valuation ring of $K$ and $|dx|$ the Haar measure,
normalized so that $\cO_K^n$ has measure one.

Similarly, if $g$ is a polynomial over $\cO[1/N]$ and with $K$ and
$\psi_K$ as above in this section, we write
$$
T_{g,K,\psi_K}(y):=\int_{\cM_K^n}\psi_K( y g(x))|dx|
$$
with $\cM_K$ the maximal ideal of $\cO_K$.

For $K=\QQ_p$ one has
$$
T_g(\frac{1}{p^m})=T_{g,\QQ_p,\exp(2\pi i\cdot)}(\frac{1}{p^m})
$$
and
$$
S_f(\frac{1}{p^m})=S_{f,\QQ_p,\exp(2\pi i\cdot)}(\frac{1}{p^m}).
$$

Write $|\cdot|_K$ for the standard norm on $K$. So, the norm of a
uniformizer of $\cO_K$ equals $\frac{1}{q_K}$ with $q_K$ the number
of elements in the residue field of $\cO_K$. Let $|\cdot|$ denote
the norm on $\CC$.

\subsection{Nondegenerate polynomials}\label{s:nondeg}

Let $f$ be a nonconstant polynomial over $\CC$ in $n$ variables with
$f(0)=0$.\footnote{When $f(0)\not=0$, then there is no harm in
replacing $f$ by $f-f(0)$: all corresponding changes in the paper
are easily made.} Write $f(x)=\sum_{i\in \NN^n} a_i x^i$ with
$a_i\in \CC$.
 The \emph{global Newton polyhedron} $\Delta_{\rm global}(f)$ of $f$
is the convex hull of the support ${\rm Supp}(f)$ of $f$, with
$${\rm Supp}(f):=\{i\mid i\in\NN^n,\, a_i\not=0\}.$$
  The \emph{Newton polyhedron $\Delta_0(f)$
of $f$ at the origin} is
$$
\Delta_0(f):=\Delta_{\rm global}(f)+\RR_+^n
$$
with $\RR_+=\{x\in\RR\mid x\geq 0\}$ and $A+B=\{a+b\mid a\in A,\
b\in B\}$ for $A,B\subset\RR^n$.
 For a subset $I$ of $\RR^n$ define
 $$f_I(x):= \sum_{i\in I\cap
\NN^n} a_i x^i.$$
  By the \emph{faces} of $I$ we mean
$I$ itself and each nonempty convex set of the form
$$
\{x\in I\mid  L(x)=0\}
$$
where $L(x)=b_0+ \sum_{i=1}^n b_ix_i$ with $b_i\in \RR$ is such that
 $L(x)\geq 0$ for each $x\in I$.
 For $\cI$ a collection of subsets of $\RR^n$, call $f$
\emph{nondegenerate with respect to $\cI$} when $f_I$ has no
critical points on $(\CC^\times)^n$ for each $I$ in $\cI$, where
$\CC^\times=\CC\setminus\{0\}$. As is standard terminology, call $f$
nondegenerate w.r.t.~$\Delta_0(f)$ when $f$ is nondegenerate
w.r.t.~the compact faces of $\Delta_0(f)$.

 For $k\in\RR^n_+$ put
\begin{eqnarray*}
\nu(k) & = & k_1+k_2+\ldots+k_n,\\
N(f)(k) & = & \min_{i\in\Delta_0(f)} k\cdot i,\\
F(f)(k) &= & \{i\in \Delta_0(f)\mid k\cdot i = N(f)(k)\},
\end{eqnarray*}
where $k\cdot i$ is the standard inproduct on $\RR^n$. Denote by
$F_0(f)$ the smallest face of $\Delta_0(f)$ which has nonempty
intersection with the diagonal $\{(t,\ldots,t)\mid t\in\RR\}$ and
let $(1/\sigma(f),\ldots,1/\sigma(f))$ be the intersection point of
the diagonal with $F_0(f)$. Let $\kappa(f)$ be the codimension of
$F_0(f)$ in $\RR^n$. If there is no possible confusion, we write
$\sigma$ instead of $\sigma(f)$, $N(k)$ instead of $N(f)(k)$, $F(k)$
for $F(f)(k)$, and $\kappa$ for $\kappa(f)$.

\subsection{Notation}\label{fg} Often in this paper, $f$ is a quasi-homogeneous
polynomial and $g$ a polynomial, both over $\cO[1/N]$ and in $n$
variables, with $\cO$ a ring of integers and $N>0$ an integer, such
that $f$ and $g$ are nonzero, $f(0)=g(0)=0$ and such that $f$ is
nondegenerate w.r.t.~$\Delta_0(f)$ and $g$ is nondegenerate
w.r.t.~$\Delta_0(g)$. If $f$ and $g$ are such we will say that they
are as in \ref{fg}.

By $K$ is usually meant a nonarchimedean local field that is an
algebra over $\cO[1/N]$ and by $q_K$ the number of elements in the
residue field of $\cO_K$. If $f$ and $g$, $K$ and $q_K$ are such we
will say that they are as in \ref{fg}.

\section{The main results}

Let $f$, $g$, and $K$ be as in \ref{fg} and use the notation of
sections \ref{ST} and \ref{s:nondeg}.

\begin{theorem}\label{mt1}
There exists $c$, only depending on $\Delta_0(f)$, resp.~on
$\Delta_0(g)$, such that for all $K$ with big enough residual
characteristic, all $\psi_K$ as in section \ref{ST}, and all $y$ in
$K$ with $\ord_K(y)<0$,
\begin{equation}\label{c2}
|S_{f,K,\psi_K}(y)|< c|y|_K^{-\sigma(f)}|\ord_K(y)|^{\kappa(f)-1},
\end{equation}
resp.
\begin{equation}\label{c3}
|T_{g,K,\psi_K}(y)|<c|y|_K^{-\sigma(g)}|\ord_K(y)|^{\kappa(g)-1}.
\end{equation}
\end{theorem}

In the case that $f$ is moreover homogeneous, a slightly weaker form
than (\ref{c2}) has been conjectured by Igusa \cite{Igusa3} and
(\ref{c2}), (\ref{c3}) have been conjectured by Denef and Sperber
\cite{DenSper}.
 The similar conjecture by Igusa mentioned in the beginning of the
introduction (namely for all primes $p$) then follows by standard
arguments, cf.~\cite{Cigumodp}. Denef and Sperber prove (\ref{c3})
when no vertex of $\Delta_0(g)$ belongs to $\{0,1\}^n$. As briefly
mentioned in the introduction, (\ref{c2}) for homogeneous $f$ is
proven in \cite{DenSper} under the condition that no vertex of
$\Delta_0(f)$ belongs to $\{0,1\}^n$ and in \cite{CDenSper} in full
generality. Recall that for (\ref{c2}), $f$ is allowed to be
quasi-homogeneous and $g$ is general.

\subsection{} Let $h$ be any quasi-homogeneous polynomial over
$\cO[1/N]$ (so, $h$ is not necessarily nondegenerate). Let
$\alpha_h$ be the motivic oscillation index of $h$ as defined in
\cite{Cigumodp} according to a suggestion by Jan Denef. As in \cite{Cigumodp}, we
conjecture that there
exists $c$ such that for all $K$ with big enough residual
characteristic (where $K$ runs over nonarchimedean local fields that are algebras over $\cO[1/N]$), all $\psi_K$ as
in section \ref{ST}, and all $y$ in $K$  with $\ord_K(y)<0$,
\begin{equation}\label{c4}
|S_{h,K,\psi_K}(y)|< c|y|_K^{\alpha_h}|\ord_K(y)|^{n-1}.
\end{equation}
Generalizing the main theorem of \cite{Cigumodp} to the
quasi-homogeneous case, we prove the following evidence for this
conjecture.
\begin{theorems}\label{mt2}
Statement (\ref{c4}) holds for all $y\in K$ of order $-1$ and $-2$.
That is, there exists $c$ such that for all $K$ of big enough
residual characteristic with $K$ an algebra over $\cO[1/N]$, all
$\psi_K$ as in section \ref{ST}, and all $y\in K$ of order $-1$ or
order $-2$, statement (\ref{c4}) holds.
\end{theorems}

Note that (\ref{c2}) of Theorem \ref{mt1} also constitutes evidence
for this conjecture.

\subsection{}
A third main result is Theorem \ref{Katzgen} on finite field
exponential sums and generalizes some results by Katz \cite{Katz} to
the quasi-homogeneous case (not necessarily nondegenerated).
Proposition \ref{c:quasinondeg} represents a new kind (a similar
bound in the homogeneous case was introduced in \cite{CDenSper}) of
very simple bounds for nondegenerate finite field exponential sums
for quasihomogeneous polynomials, based on the combinatorics of the
Newton Polyhedron.

\subsection{}
The underlying idea of the proofs of these questions, apart from the
ideas and results of \cite{DenSper}, \cite{Katz}, \cite{CDenSper},
and \cite{Cigumodp}, is that one can make a homogeneous polynomial
out of a quasi-homogeneous polynomial $f(x)$ in finitely many steps,
by replacing one of the variables $x_i$ by $yx_i$ with $y$ a new
variable, that is, by replacing $f(x)$ by
$$f_1(x,y):=f(x_1,x_2,\ldots, x_{i-1},x_iy,x_{i+1},\ldots,x_n)$$ and so
on, and then one compares the considered objects for $f$ and for
$f_1$.

\section{A Denef - Sperber Formula for $S_{f,K,\psi_K}$ and $T_{g,K,\psi_K}$}

The following proposition has essentially the same proof as
Proposition (2.1) of \cite{DenSper}, but is slightly more general.
We give the proof for the convenience of the reader.

\begin{prop}\label{r21} Let $f$, $g$, $K$, and $q_K$ be as in \ref{fg}. There
exists $M>0$ such that if the residual characteristic of $K$ is
bigger than $M$, then for all $y$ in $K$ with $\ord_K(y)<0$,
 \begin{equation}\label{DF}
S_{f,K,\psi_K}(y)=d_K\cdot\sum_{\small \mbox{$\tau$ face of
$\Delta_0(f)$}}\big( A_f(K,y,\tau)
 +
 E(K,y,f_\tau,\psi_K) B_f(K,y,\tau)
 \big)
\end{equation}
and
 \begin{equation}\label{DF2}
T_{g,K,\psi_K}(y)=d_K\cdot\sum_{\sur{\mbox{$\tau$ compact}} {\mbox{
face of $\Delta_0(g)$}}}\big( A_g(K,y,\tau)
 +
 E(K,y,g_\tau,\psi_K) B_g(K,y,\tau)
 \big),
\end{equation}
 with
$$
d_K:=(1-q_K^{-1})^n,
$$
$$
A_f(K,y,\tau):=\sum_{\small
\begin{array}{c}k\in \NN^n
\\ F(f)(k)=\tau\\ N(f)(k)\geq -\ord(y) \end{array}}
q_K^{-\nu(k)},
$$
$$
B_f(K,y,\tau):= \sum_{\small
\begin{array}{c}k\in \NN^n
\\ F(f)(k)=\tau\\ N(f)(k)= -\ord(y)-1 \end{array}}q_K^{-\nu(k)},
$$
and similarly for $A_g(K,y,\tau)$ and $A_g(K,y,\tau)$, and for
$h(x)$ either $f_\tau$ or $g_\tau$,
\begin{equation}\label{Ept.}
E(K,y,h,\psi_K)=\frac{1}{(q_K-1)^n} \sum_{\small \mbox{$x\in
(\GG_m(\FF_{q_K}))^n$}}\varphi_{y}(h(x)),
\end{equation}
with $\varphi_{y}:\FF_{q_K} \to\CC^\times$ a nontrivial additive
character on $\FF_{q_K}$ sending $b$ to $\psi_K(y \pi_K^{-\ord
y-1}b')$, with $b'$ a representant in $\cO_K$ of $b$ and $\pi_K$ a
uniformizer of $\cO_K$.

\end{prop}
\begin{proof}
Rewrite the definition
$$
S_{f,K,\psi_K}(y)=\int_{\cO_K^n}\psi_K( y f(x))|dx|
$$
as
$$
S_{f,K,\psi_K}(y) = \sum_{\small \mbox{$\tau$ face of
$\Delta_0(f)$}}\sum_{\small
\begin{array}{c}k\in \NN^n
\\ F(f)(k)=\tau\end{array}}
 \int_{\ord_K\, x=k}\psi_K(y f(x))|dx|.
$$
Put $x_j=\pi_K^{k_j}u_j$ for $k\in\NN^n$, with $\pi_K$ a uniformizer
of $\cO_K$. Then $|dx|=q_K^{-\nu(k)}|du|$ and
$$
f(x)=\pi_K^{N(f)(k)}\big(f_{F(f)(k)}(u)+ \pi_K(...)\big),
$$
for $x$ with $\ord x=k$, where the dots take values in $\cO_K$.
Hence, $S_{f,K,\psi_K}(y)$ equals
\begin{equation}\label{ein}
\sum_{\small \mbox{$\tau$ face of $\Delta_0(f)$}}\sum_{\small
\begin{array}{c}k\in \NN^n
\\ F(f)(k)=\tau\end{array}}q_K^{-\nu(k)} \int_{u\in(\cO_K^\times)^n}
\psi_K\big( y \pi_K^{N(f)(k)}(f_\tau(u)+ \pi_K(...))\big)|du|,
\end{equation}
with $\cO_K^\times$ the group of units in $\cO_K$. Because of the
nondegenerateness assumptions, for $\tau$ any face of $\Delta_0(f)$
and when the residue field characteristic of $K$ is big enough, the
reduction $f_\tau\bmod \cM_K$ has no critical points on
$(\FF_{q_K}^\times)^n$ (this holds indeed for all faces and not only
for the compact faces by the quasi-homogeneity of $f$). Hence, the
integral in (\ref{ein}) is zero whenever $\ord_K(y)+N(f)(k)\leq -2$.
 When $\ord_K(y)+N(f)(k)\geq 0$, the integral over $(\cO_K^\times)^n$ in
(\ref{ein}) is just the measure of $(\cO_K^\times)^n$ and thus
equals $(1-q_K^{-1})^n$. When $\ord_K(y)+N(f)(k)=-1$, the integral
over $(\cO_K^\times)^n$ in (\ref{ein}) equals $q_K^{-n}(q_K-1)^n
E(K,y,h,\psi_K)$.
 Equation (\ref{DF}) now follows from the fact that $S_{f,K,\psi_K}(y)$ equals (\ref{ein}).

Equation (\ref{DF2}) follows in a similar way. Namely, one writes
$$
T_{g,K,\psi_K}(y) = \sum_{\small \mbox{$\tau$ face of
$\Delta_0(g)$}}\sum_{\small
\begin{array}{c}k\in (\NN\setminus\{0\})^n
\\ F(g)(k)=\tau\end{array}}
 \int_{\ord_K\, x=k}\psi_K(y g(x))|dx|,
$$
which equals
$$
T_{g,K,\psi_K}(y) = \sum_{\small \mbox{$\tau$ compact face of
$\Delta_0(g)$}}\sum_{\small
\begin{array}{c}k\in \NN^n
\\ F(g)(k)=\tau\end{array}}
 \int_{\ord_K\, x=k}\psi_K(y g(x))|dx|,
$$
and one proceeds as for $f$.

\end{proof}

\section{Bounds for $\nu(k)$, $A_f$, $B_f$, $A_g$, and $B_g$}

We recall two result from \cite{CDenSper}, about lower bounds for
$\nu(k)$ in terms of $N(f)(k)$ and $N(g)(k)$, and upper bounds for
$A_f$, $B_f$, $A_g$, and $B_g$.

\begin{prop}[\cite{CDenSper}, Theorem 4.1]\label{len}
Let $h$ be any nonzero polynomial in $n$ variables with $h(0)=0$.
Let $\tau$ be a face of $\Delta_0(h)$. Then one has for all $k$ in
$\NN^n$ with $F(h)(k)=\tau$ that
\begin{equation}\label{enu0}
\nu(k)\geq \sigma(h)\big(N(h)(k) + 1\big) - \sigma(h_\tau).
\end{equation}
\end{prop}
The main point in Proposition \ref{len}, and one of the main
differences with the approach from \cite{DenSper}, is that one
subtracts $\sigma(h_\tau)$ on the right hand side of (\ref{enu0}).
Subtracting $\sigma(h)$ would yield trivial bounds since one has
$\nu(k)\geq \sigma(h) N(h)(k)$ for all $k\in \NN^n_+$.

The bounds for $A_f$ and $A_g$ in Proposition \ref{AB} are
essentially proven in \cite{DenSper} and follow from Lemma (3.3) of
\cite{DenSper}, which is recalled in Lemma \ref{l3.3} below (in
\cite{CDenSper} the proof of these bounds is repeated from
\cite{DenSper}). The bounds for $B_f$ and $B_g$ are essentially
proven in \cite{CDenSper}, and follow from Lemma (3.3) of
\cite{DenSper} and Proposition \ref{len}.

\begin{prop}[\cite{CDenSper}, \cite{DenSper}]\label{AB}
Let $f$, $g$, $A$, and $B$ be as in Proposition \ref{r21}. Then
there exists a real number $c>0$ such that for all $K$ as in
\ref{fg}, all faces $\tau$ of $\Delta_0(f)$, resp.~of $\Delta_0(g)$,
and all $y$ in $K$ with $\ord_K(y)<0$,
\begin{equation}\label{eAe}
A_f(K,y,\tau)\leq c|y|_K^{-\sigma(f)}|\ord_K(y) |^{\kappa(f)-1}
\end{equation}
and
\begin{equation}\label{eBe}
B_f(K,y,\tau)\leq c|y|_K^{-\sigma(f)} q_K^{
\sigma(f_\tau)}|\ord_K(y) |^{\kappa(f)-1},
\end{equation}
resp.,
\begin{equation}\label{eCe}
A_g(K,y,\tau)\leq c|y|_K^{-\sigma(g)}|\ord_K(y) |^{\kappa(g)-1}
\end{equation}
and
\begin{equation}\label{eDe}
B_g(K,y,\tau)\leq c|y|_K^{-\sigma(g)}q_K^{ \sigma(f_\tau)}|\ord_K(y)
|^{\kappa(g)-1}.
\end{equation}
Moreover, one can choose $c$ depending on $\Delta_0(f)$ and
$\Delta_0(g)$ only.
\end{prop}
\begin{proof}
We give the proof for $B_g$ for the convenience of the reader. The
proof for $B_f$ is similar and the proofs for $A_f$ and $A_g$ can be
found in \cite{DenSper}. To derive (\ref{eDe}) from Lemma \ref{l3.3}
and Proposition \ref{len}, use for $C$ the topological closure of
the convex hull of $\{0\}\cup \{k\in\NN^n\mid F(g)(k)= \tau\}$, and
note that $C^{\rm int}\cap\NN^n= \{k\in\NN^n\mid F(g)(k)= \tau\}$.
Clearly $\kappa(g)\geq 1$ and $\kappa(g)\geq \dim \{k\in C \mid
\nu(k)=N(g)(k)\sigma\}$. By (\ref{enu0}), $\nu(k)\geq
N(g)(k)\sigma(g)+ \sigma(g) - \sigma(g_\tau)$ for all $k\in C^{\rm
int}\cap\NN^n$. So for $\gamma$ one can take
$\sigma(g)-\sigma(g_\tau)$ which is nonnegative.
\end{proof}

\begin{lem}[\cite{DenSper}, Lemma (3.3)]\label{l3.3}
Let $C$ be a convex polyhedral cone in $\RR_+^n$ generated by
vectors in $\NN^n$, and let $L$ be a linear form in $n$ variables
with coefficients in $\NN$. We denote by $C^{\rm int}$ the interior
of $C$ in the sense of Newton polyhedra. Let $\sigma>0$ and
$\gamma\geq 0$ be real numbers satisfying
\begin{equation}
\nu(k)\geq L(k)\sigma + \gamma,\ \mbox{ for all }k\in C^{\rm
int}\cap \NN^n.
\end{equation}
Put
$$
e=\dim \{k\in C \mid \nu(k)=L(k)\sigma\}.
$$
Then there exists a real number $c>0$ such that for all $m\in\NN$
and for all $q\in\RR$ with $q\geq 2$,
\begin{equation}
\sum_{\begin{array}{c}k\in C^{\rm int}\cap\NN^n
\\ L(k)=m\end{array}} q^{-\nu(k)}\leq
cq^{-m\sigma-\gamma}(m+1)^{\max(0,e-1)}.
\end{equation}
\end{lem}

\section{Estimates for finite field exponential sums}\label{sec:Katz}

For each big enough prime $p$, let $\psi_p$ be a nontrivial additive
character from $\FF_p$ to $\CC^\times$ and for each power $q$ of
$p$, let $\psi_q$ be the additive character from $\FF_q$ to
$\CC^\times$ which is the composition of $\psi_p$ with the trace of
$\FF_q$ over $\FF_p$.

\begin{lem}\label{sigfg}
Let $f$ be a polynomial over some number field in the $n$ variables
$x_1,\ldots,x_n$. (Thus $f$ is not necessarily quasi-homogeneous.)
Let $g(x_1,\ldots,x_n,y)$ be the polynomial $f(x_1y,x_2,\ldots,x_n)$
in the $n+1$ variables $(x,y)$. Suppose that $f(0)=g(0)=0$. Then
$\sigma(f)=\sigma(g)$. Moreover, if $f$ is nondegenerate
w.r.t.~$\Delta_0(f)$, then $g$ is nondegenerate
w.r.t.~$\Delta_0(g)$.
\end{lem}
\begin{proof}
Clearly, for any point $P=(p_{1},\ldots,p_n)$ on $\Delta_0(f)$, the
point $P'=(p_1,\ldots,p_n,p_1)$ lies on $\Delta_0(g)$. Indeed, this
holds for points $P$ in the support of $f$, thus also for $P$ in the
convex hull of the support of $f$, and thus for general $P$ in
$\Delta_0(f)$. Vice versa, for any point $Q=(q_1,\ldots,q_{n+1})$ on
$\Delta_0(g)$, the point $Q'=(q_1,\ldots,q_n)$ lies in
$\Delta_0(f)$. From this follows that $\sigma(f)=\sigma(g)$. The
statement about the nondegenerateness is immediate since we
performed a coordinate transformation on the torus $\GG_m^{n+1}$
induced by a module isomorphism of $\ZZ^{n+1}$.
\end{proof}

The following generalizes Corollary 6.4 of \cite{CDenSper} from the
homogeneous case to the quasi-homogeneous case. Corollary 6.4 of
\cite{CDenSper} is proven in \cite{CDenSper} using results by Katz
\cite{Katz}, by Segers \cite{Segers}, and by the author
\cite{Cigumodp}.

\begin{prop}\label{c:quasinondeg}
Let $f$ be as in \ref{fg}. (In particular, $f$ is
quasi-homogeneous.) Then there exists $a>0$ such that for each big
enough prime $p$, each power $q$ of $p$ such that $\FF_q$ is an
algebra over $\cO[1/N]$, one has
\begin{equation}\label{quasi1}
|\frac{1}{(q-1)^n}\sum_{x\in\GG_m^n(\FF_q)} \psi_q(f(x))| < a
q^{-\sigma(f)}
\end{equation}
and
\begin{equation}\label{quasi2}
|\frac{1}{q^n}\sum_{x\in\AA^n(\FF_q)} \psi_q(f(x))| < a
q^{-\sigma(f)},
\end{equation}
for all choices of $\psi_p$. Moreover, $a$ can be chosen depending
on $\Delta_0(f)$ only.
\end{prop}
\begin{proof}
We first prove (\ref{quasi1}). First suppose that $f$ is
homogeneous. If moreover $f(x)$ is of the form
$\sum_{i=1}^{n}a_ix_i$, (\ref{quasi1}) follows from the obvious
inequality $\sigma(f)\leq n$. If $f$ is nonlinear (that is, $f(x)$
is not of the form $\sum_{i=1}^{n}a_ix_i$) and homogeneous, then the
proposition holds by Corollary 6.4 of \cite{CDenSper}, see also
section 9 of \cite{CDenSper}. So, the case of homogeneous $f$ is
done. Now let $f$ be quasi-homogeneous. After finitely many steps as
from going from $f$ to $g$ in Lemma \ref{sigfg} and renumbering the
variables $x_i$, one can go from a quasi-homogeneous polynomial to a
homogeneous polynomial $h$ in possibly more variables. Going from
$f$ to $g$ as in Lemma \ref{sigfg} consists of two steps: first one
defines a polynomial $f_0(x,y)$ which equals $f(x)$, that is, one
considers $f$ as a polynomial in one more variable $y$. Then one
performs the transformation on the torus, $(x,y)$ to
$(x_1y,x_2,\ldots,x_n)$, to obtain $g$. Under both these steps, the
sum (\ref{quasi1}) remains unaltered, since one divides by $(q-1)^n$
with $n$ the number of variables in the left hand side of
(\ref{quasi1}), and since in the second step we just perform a
transformation on the torus. Lemma \ref{sigfg} together with the
homogeneous case now proves (\ref{quasi1}).

The inequality (\ref{quasi2}) is not used in this paper and follows
from (\ref{c2}) of Theorem \ref{mt1} with argument $y$ of order
$-1$.
\end{proof}

\section{A corollary of results by Katz}

The results in this section will only serve to prove Theorem
\ref{mt2}, not to prove Theorem \ref{mt1}.  In this section as well
as in Theorem \ref{mt2} we focus on quasi-homogeneous polynomials in
general (thus not necessarily non degenerated ones).

Call a collection of polynomials $f_i$ in $n$ variables
\emph{quasi-homogeneous with the same weights} if there are integers
$a_j>0$, not depending on $i$, such that
$f_i(x_1^{a_1},\ldots,x_n^{a_n})$ is homogeneous for each $i$.

\begin{lem}\label{intersect}
Let $f_i$ be quasi-homogeneous polynomials with the same weights,
for $i=1,\ldots,d$. Suppose that $f_d$ is nonconstant. Let $X$ be
the locus of the $f_i$ for $i=1,\ldots,d-1$ in $\AA^n$ and let $Y$
be the locus of the $f_i$ for $i=1,\ldots,d$ in $\AA^n$. Then
$$
\dim X\leq \dim Y+1,
$$
where the dimension of the empty scheme is $-1$.
\end{lem}
\begin{proof}
Clearly this statement holds when the $f_i$ are homogeneous
polynomials, by intersection theory. Since the $f_i$ have the same
weights, let $a_{j}$, $j=1,\ldots,n$ be positive integers such that
each of the $f_i(x_1^{a_1},\ldots,x_n^{a_n})$ is homogeneous.
Consider the application
$$
G:\AA^n\to\AA^n:x\mapsto (x_1^{a_1},\ldots,x_n^{a_n}).
$$
The application $G$ is finite to one, hence $G^{-1}(X)$ has the same
dimension as $X$, and $G^{-1}(Y)$ has the same dimension as $Y$. By
the above statement for homogeneous polynomials, the dimension of
$G^{-1}(X)$ is less or equal than the dimension of $G^{-1}(Y)$ plus
1. The lemma is proven.
\end{proof}

\begin{cor}\label{lem:fg}
Let $f$ be a quasi-homogeneous polynomial over a field of
characteristic zero in the $n$ variables $x_1,\ldots,x_n$. Let
$g(x_1,\ldots,x_n,y)$ be the polynomial $f(x_1y,x_2,\ldots,x_n)$ in
$n+1$ variables. Let $C_f$ be the closed subscheme of $\AA^n$ given
by $\grad f=0$, and let $C_g$ be the closed subscheme of $\AA^{n+1}$
given by $\grad g=0$. Suppose that $C_f$ contains the point $0$.
Then
$$
\dim C_g\leq \dim C_f+1,
$$
where the dimension of the empty scheme is said to be $-1$.
\end{cor}
\begin{proof}

For $i=1,\ldots, n$ let $f_i(x)$ be the polynomial $\frac{\partial
f}{\partial x_i}(x)$, let $g_i(x,y)$ be the polynomial
$\frac{\partial g}{\partial x_i}(x,y)$, and let $g_{0}(x,y)$ be
$\frac{\partial g}{\partial y}(x,y)$. Note that the polynomials
$f_i$ are quasi-homogeneous with the same weights. Similarly, the
polynomials $g_i$ are quasi-homogeneous with the same weights.

By the chain rule for differentiation one has
$$
g_0(x,y) = x_1f_1(x_1y,x_2,\ldots,x_n),
$$
$$
g_1(x,y) = y f_1(x_1y,x_2,\ldots,x_n),
$$
and, for $i>1$,
$$
g_i(x,y) = f_i(x_1y,x_2,\ldots,x_n).
$$

On the part $U_1$ of $\AA^{n+1}$ where $x_1y\not= 0$  the bound is
clear. Indeed, $U_1\cap C_g$ is the subscheme of $\AA^{n+1}$ given
by $x_1y\not= 0$ and $f_i(x_1y,x_2,\ldots,x_n)=0$ for
$i=1,\ldots,n$, which has dimension at most $\dim C_f+1$.

Now work on the part $U_2$ where $x_1=0$ and $y\not=0$. Then
$U_2\cap C_g$ is the subscheme of $\AA^{n+1}$ given by $f_i(x)=0$
for $i=1,\ldots,n$, $y\not=0$, and $x_1=0$, which has dimension at
most
 $\dim C_f+1$.

Finally work on the part $U_3$ where $x_1=0$ and $y=0$. Then
$U_3\cap C_g$ is the subscheme of $\AA^{n+1}$ given by $x_1=y=0$ and
$0=f_i(x)$ for $i=2,\ldots,n$. We know that $f_1$ is nonconstant
since $C_f$ contains $0$. So, the subscheme of $\AA^n$ given by
$0=f_i(x)$ for $i=2,\ldots,n$ has at most dimension $\dim C_f+1$ by
Lemma \ref{intersect}. Hence, the dimension of $C_g\cap U_3$ has at
most dimension $\dim C_f+1$ and the corollary is proven.
\end{proof}

The following is a trivial lemma.
\begin{lem}\label{trivlem}
Let $f$ be a nonconstant, quasi-homogeneous polynomial in $n$
variables $x_1,\ldots,x_n$ over a field $k$ of characteristic zero.
Then exactly one of the following two statements holds:  either $0$
is a critical point of $f$, or $f$ contains a term of the form
$a_ix_i$ with $a_i\not=0$ in $k$ for some $i\in\{1,\ldots,n\}$.
\end{lem}

From Lemma \ref{intersect} and Corollary \ref{lem:fg}, Katz' results
of \cite{Katz} can be generalized to quasi-homogeneous polynomials,
see Theorem  \ref{Katzgen}. Namely, Theorem \ref{Katzgen}
generalizes Katz' results \cite{Katz} and some of their corollaries
in \cite{CDenSper} from homogeneous to quasi-homogeneous
polynomials, on $\GG_m^n$ and on $\AA^n$.

\begin{theorem}\label{Katzgen}
Let $f$ be a quasi-homogeneous polynomial in $n$ variables over
$\cO[1/N]$ for some ring of integers $\cO$ and integer $N>0$.
Suppose that $0$ is a critical point of $f$. Let $d$ be the
dimension of the locus of $\grad f=0$, with the dimension of the
empty scheme equal to $-1$. Then there exists $a>0$ such that for
each big enough prime $p$, each power $q$ of $p$ such that $\FF_q$
is an algebra over $\cO[1/N]$, one has
\begin{equation}\label{quasi0}
|\frac{1}{(q-1)^n}\sum_{x\in\GG_m^n(\FF_q)} \psi_q(f(x))| < a
q^{\frac{-n+d}{2}}
\end{equation}
and
\begin{equation}\label{quasi00}
 |\frac{1}{q^n}\sum_{x\in\AA^n(\FF_q)} \psi_q(f(x))| < a
q^{\frac{-n+d}{2}},
\end{equation}
for all choices of $\psi_p$.
\end{theorem}
\begin{proof}
We first prove (\ref{quasi0}). By \cite{CDenSper}, Corollary 6.1 and
section 9, (\ref{quasi0}) holds for nonlinear homogeneous $f$. Now
let $f$ be quasi-homogeneous and suppose that $0$ is a critical
point. After finitely many steps as from going from $f$ to $g$ in
the proof of Corollary \ref{lem:fg} and after renumbering the
coordinates, one can go from a quasi-homogeneous polynomial to a
nonlinear homogeneous polynomial $h$ in possibly more variables.
Going from $f$ to $g$ consists of two steps: first one defines a
polynomial $f_0(x,y)$ which equals $f(x)$, that is, one considers
one more variable $y$. Then one performs the transformation on the
torus $(x,y)$ to $(x_1y,x_2,\ldots,x_n)$ to obtain $g$. Under both
these steps, the sum which is the left hand side of (\ref{quasi0})
remains unaltered, since one divides by $(q-1)^n$ with $n$ the
number of variables in the left hand side of (\ref{quasi0}), and
since one just performs a transformation on the torus. Corollary
\ref{lem:fg} now proves (\ref{quasi0}).

Now let us prove (\ref{quasi00}).
 Let $f_0(\hat x)$ be the polynomial $f(0,\hat x)$ in the
$n-1$ variables $\hat x = (x_2,\ldots,x_n)$. Clearly $f_0$ is
quasi-homogeneous in $n-1$ variables. Write $C_f$ for the locus of
$\grad f=0$ in $\AA^{n}$ and $C_{f_0}$ for the locus of $\grad
f_0=0$ in $\AA^{n-1}$. By (\ref{quasi0}) it is enough to show that
$n-1+\dim C_{f_0}\leq n+d$. Thus we only have to show that
\begin{equation}\label{toshow}
\dim C_{f_0}\leq \dim C_{f}+1,
\end{equation}
where we recall that $\dim C_{f}=d$. This inequality (\ref{toshow})
follows from writing
$$
f(x)=x_1\tilde f(x)+f_0(\hat x)
$$
with $\tilde f$ a polynomial in $x$, and from Lemma \ref{intersect},
as follows.
 By Lemma \ref{trivlem}, $\tilde f$ is
nonconstant. For $i=1,\ldots, n$, let $f_i(x)$ be the polynomial
$\frac{\partial f}{\partial x_i}(x)$ and for $j=2,\ldots,n$, let
$f_{0j}(\hat x)$ be the polynomial $\frac{\partial f_0}{\partial
x_j}(\hat x)$. Let $Y$ be the locus in $\AA^n$ of the polynomials
$f_{0j}$ and the equation $x_1=0$. Let $Z$ be the intersection of
$C_f$ with $x_1=0$. Then clearly
 \begin{equation}\label{dimYf0}
 \dim Y=\dim C_{f_0}.
\end{equation}
  Since $f_1=\tilde f
+ x_1\partial \tilde f /
\partial x_1$, it follows that $Z$ equals the intersection of $Y$
with $\tilde f(x)=0$. The polynomials $f_i$ and $\tilde f$ are
quasi-homogeneous with the same weights. By Lemma \ref{intersect}
and since $\tilde f$ is nonconstant,  $\dim Y \leq \dim Z+1$.
Clearly also $\dim Z \leq d$. Now (\ref{toshow}) follows by
(\ref{dimYf0}) and thus the Proposition is proved.
\end{proof}

\section{Proofs of the main theorems}

\begin{proof}[Proof of Theorem \ref{mt1}]
By (\ref{quasi1}) of Proposition \ref{c:quasinondeg}, one finds
immediately that there is $c$ such that for all $K$ and $q_K$ as in
\ref{fg} of big enough residue characteristic, and all faces $\tau$
of $\Delta_0(f)$, resp.~all compact faces $\tau$ of $\Delta_0(g)$,
\begin{equation}\label{Esig}
|E(K,y,f_\tau,\psi_K)| <cq_K^{-\sigma(f_\tau)},
\end{equation}
resp.
\begin{equation}\label{Esigb}
|E(K,y,g_\tau,\psi_K)| <cq_K^{-\sigma(g_\tau)},
\end{equation}
since automatically each $f_\tau$, resp.~each $g_\tau$, is
quasi-homogeneous for such $\tau$. Moreover, $c$ only depends on
$\Delta_0(f)$, resp.~on $\Delta_0(g)$, by Proposition
\ref{c:quasinondeg}.
 Now use
Proposition \ref{r21}, Proposition \ref{AB}, and (\ref{Esig}),
(\ref{Esigb}).
\end{proof}

The above proof of Theorem \ref{mt1} is similar to the one of the
main theorem of \cite{CDenSper}. The key new ingredient in its proof
is Proposition \ref{c:quasinondeg} for quasi-homogeneous polynomials
instead of for homogeneous polynomials. Note that Proposition
\ref{c:quasinondeg} was used both for (\ref{c2}) and for (\ref{c3}).

\par

\begin{proof}[Proof of Theorem \ref{mt2}]\label{igum=1}
By Lemma \ref{trivlem} we may suppose that $0$ is a critical point
of $f$. Otherwise, that is, when $0$ is not a critical point, one
has $S_{h,K,\psi_K}(y)=0$ for the $K$ and $y$ under consideration.
The theorem then follows from (\ref{quasi00}) of Theorem
\ref{Katzgen} and from the fact that the motivic oscillation index
$\alpha_h$ of $h$ satisfies
$$
\frac{-n+d}{2}\leq \alpha_h,
$$
with $d$ the dimension of the locus of $\grad h=0$, by Theorem 5.1
of \cite{Cigumodp}.
\end{proof}

\subsection*{Acknowledgment}
I would like to thank J.~Denef, E.~Hrushovski, F.~Loeser, and
especially J.~Nicaise for inspiring discussions during the
preparation of this paper.

\bibliographystyle{amsplain}
\bibliography{anbib}

\providecommand{\bysame}{\leavevmode\hbox to3em{\hrulefill}\thinspace}
\providecommand{\MR}{\relax\ifhmode\unskip\space\fi MR }
\providecommand{\MRhref}[2]{%
  \href{http://www.ams.org/mathscinet-getitem?mr=#1}{#2}
}
\providecommand{\href}[2]{#2}
\begin{thebibliography}{1}

\bibitem{CDenSper}
R.~Cluckers, \emph{{I}gusa - {D}enef - {S}perber conjecture on nondegenerate
  $p$-adic exponential sums}, to appear in Duke Math. J.,
  arXiv:math.NT/0606269.

\bibitem{Cigumodp}
\bysame, \emph{{I}gusa's conjecture on exponential sums modulo $p$ and $p^2$
  and the motivic oscillation index}, to appear in Int. Math. Res. Not. IMRN,
  arXiv:math.NT/0602438.

\bibitem{DenSper}
J.~Denef and S.~Sperber, \emph{Exponential sums mod $p^n$ and {N}ewton
  polyhedra}, Bull. Belg. Math. Soc. Simon Stevin \textbf{suppl.} (2001),
  55--63.

\bibitem{Igusa3}
J.~Igusa, \emph{Lectures on forms of higher degree (notes by {S}. {R}aghavan)},
  Lectures on mathematics and physics, Tata institute of fundamental research,
  vol.~59, Springer-Verlag, 1978.

\bibitem{Katz}
N.~Katz, \emph{Estimates for "singular" exponential sums}, Int. Math. Res. Not.
  IMRN (1999), no.~16, 875--899.

\bibitem{Segers}
D.~Segers, \emph{Lower bound for the poles of {I}gusa's {$p$}-adic zeta
  functions}, Math. Ann. \textbf{336} (2006), no.~3, 659--669.

\end{thebibliography}
\end{document}